\documentclass[9pt,shortpaper,twoside,web]{ieeecolor}
\usepackage{cite}
\usepackage[pdftex]{graphicx}

\def\BibTeX{{\rm B\kern-.05em{\sc i\kern-.025em b}\kern-.08em
    T\kern-.1667em\lower.7ex\hbox{E}\kern-.125emX}}
    
\title{Solution of the Continuous Time Bilinear Quadratic Regulator Problem by Krotov's Method}
\author{Ido~Halperin%
\thanks{Manuscript received ??? (Corresponding author: Ido Halperin.)}
\thanks{I. Halperin is with the Department of Civil Engineering, Ariel University, Ariel, 40700 Israel (email: idoh@ariel.ac.il)}%
}
\usepackage{generic}
\usepackage{amsmath,amssymb,amsfonts}
\usepackage[caption=false]{subfig}

\usepackage{framed}
\usepackage{textcomp}

\newtheorem{definition}{Definition}[section]
\newtheorem{theorem}{Theorem}[section]
\newtheorem{lemma}{Lemma}[section]
\usepackage[bookmarks,colorlinks,citecolor=black,linkcolor=black]{hyperref}	

\usepackage{bm,upgreek}
\usepackage{framed}

\newcommand{\Eqref}[1]{Eq. \eqref{#1}}
\newcommand{\figref}[1]{Fig. \normalfont{\ref{#1}}}
\newcommand{\ton}{\,\left[ \mathrm{ton}\right] }
\newcommand{\e}[1]{\cdot10^{#1}}
\newcommand{\nOmt}{\,\left[ \mathrm{N/m} \right] }
\newcommand{\aliu}[1]{\begin{align*}#1\end{align*}}
\newcommand{\mrb}[1]{\mathrm{\mathbf{#1}}} 
\newcommand{\aM}{\mrb{M}}	
\newcommand{\aK}{\mrb{K}}	
\newcommand{\mtrx}[1]{\begin{bmatrix}#1\end{bmatrix}}
\newcommand{\aC}{\mrb{C}}	
\newcommand{\diag}{\mathrm{\mathbf{diag}}}	

\newcommand{\kgOsec}{\,\left[ \mathrm{\frac{kg}{s}} \right] }
\newcommand{\ud}[0]{\mathrm{d}\,}
\newcommand{\mrbb}[1]{\mathrm{\boldsymbol{#1}}} 
\newcommand{\bR}{\mathbb{R}}
\newcommand{\eqgrp}[1]{\begin{split}#1\end{split}}  
\newcommand{\ali}[1]{\begin{align}#1\end{align}}

\newcommand{\ag}{\mrb{g}}
\newcommand{\ap}{\mrb{p}}
\newcommand{\az}{\mrb{z}}
\newcommand{\aI}{\mrb{I}}
\newcommand{\aZe}{\mrb{0}}

\newcommand{\Eqsref}[1]{Eqs. \eqref{#1}}

\newcommand{\aA}{\mrb{A}}	
\newcommand{\aB}{\mrb{B}}	
\newcommand{\af}{\mrb{f}}	
\newcommand{\aH}{\mrb{H}}	
\newcommand{\aQ}{\mrb{Q}}	
\newcommand{\aP}{\mrb{P}}	
\newcommand{\aN}{\mrb{N}}	
\newcommand{\aR}{\mrb{R}}	
\newcommand{\au}{\mrb{u}}	

\newcommand{\ax}{\mrb{x}}	
\newcommand{\ay}{\mrb{y}}	
\newcommand{\aw}{\mrb{w}}	
\newcommand{\agam}{\mrbb{\gamma}}	
\newcommand{\axi}{\pmb{\upxi}}	
\newcommand{\anu}{\pmb{\upnu}}	
\newcommand{\aPsi}{\mrbb{\Psi}}	

\newcommand{\cU}{\mathcal{U}}	
\newcommand{\cX}{\mathcal{X}}	

\newenvironment{myremarks}{\nopagebreak\paragraph*{\textbf{Remarks.}}\nopagebreak\begin{itemize}}{\end{itemize}}

\usepackage[normalem]{ulem}

\begin{document}

\maketitle
\begin{abstract}
This work contributes to the field of optimal control of bilinear systems. 
It concerns a continuous time, finite dimensional, bilinear state equation with a quadratic performance index to be minimized. The state equation is non-autonomous and comprises a deterministic, a-priori known excitation. The control trajectory is constrained to an admissible set without a specific structure. The performance index is a functional, quadratic  in the state variables and control signals. 
Krotov's method is used for solving this problem by means of an improving sequence. To this end, the required sequence of an improving functions is formulated. Finally, the solution is encapsulated in an algorithm form and a numerical example of structural control problem is provided.
\end{abstract}


\section{Introduction}
Bilinear state-space models are simple nonlinear models, useful for capturing dynamic attributes of systems in various fields, such as 
quantum mechanics \cite{11Krot},
chaotic dynamics \cite{08Pardalos},
biology \cite{95Aganovich,05Reluga,18Danane},
mechanical damping \cite{15Khursh,15Brez,18Pisarski,17Michaj} and structural control \cite{17Hal01,17Hal02,17Hal03,19Hal,19Hal02}. 
Even when more complex nonlinear plants are addressed, they can sometime be well approximated by bilinear models\cite{96Wang,11Krot,20Hal_EarlyAccess}.\\

The ongoing research, conducted on such systems, has been yielding diverse results, including optimal control design tools. 
Among the published works, discussing finite dimensional bilinear systems, one can find results concerning
homogeneous \cite{88Hofer,97Chanane,05Tang,05Lee,05Sergeev,17Hal01,17Hal02,17Hal03} or inhomogeneous \cite{18Wang,19Hal,19Boukhari} plants,
continuous  \cite{88Hofer,97Chanane,05Lee,17Hal01,17Hal02,17Hal03,18Wang,19Boukhari} or discrete \cite{05Tang,05Sergeev,20Hal_EarlyAccess} time,
problems with control constraints \cite{05Sergeev,17Hal01,17Hal02,17Hal03,19Hal,19Hal02,19Boukhari,20Hal_EarlyAccess} and 
quadratic \cite{88Hofer,05Tang,05Lee,05Sergeev,17Hal01,17Hal02,19Boukhari} or biquadratic \cite{17Hal03,19Hal02} performance index.
These solutions are provided in a form of algorithm. Some furnish solutions that meet necessary optimality conditions, such as Pontryagin's minimum principle \cite{18Wang,17Hal01,17Hal02,17Hal03,19Hal,19Hal02}, while others rely on sufficient conditions such as Hamilton-Jacobi-Bellman equation \cite{08Adhyaru,17Breiten} and two-sided algorithms \cite{05Sergeev}.\\

In the sixties of the previous century, series of intriguing results on sufficient conditions for global optimum of optimal control problems, have began to be published by V. F. Krotov \mbox{\cite{88Krot}}. One of the outcomes of his work is a method for solving optimal control problems---``a global method of successive improvements of control'', or in its contemporary name---\textit{Krotov's method}. 
For illustrating the gist of it one can take as an example Lyapunov's direct method for stability.
Even though the method outlines properties for the required Lyapunov function, additional effort should be exerted for finding one. 
In this sense, Krotov's method and Lyapunov's method take similar approaches.
Basically, Krotov's method successively improves the control feedback until convergence is reached. However, for this to happen a special sequence of functions should be formulated.
This subproblem stands at the method's heart and is still, in general, an  open problem that should be solved in the context of a given case. The method has several benefits. First, each iteration is proven to furnish a process whose performance is either equal or better than its former. Second, it is suitable for problems where small variations in the control are not allowed, e.g. when the set of admissible control vectors is closed. Third, the resulting control is formulated in a feedback form \cite{95Krot}. 
This is unlike methods inspired by Pontryagin's minimum principle \cite{18Wang}, which are more convenient for obtaining  open loop solutions \cite{95Krot}.\\

Krotov's method was used for optimal control of bilinear systems in several works. For example, an optimal control of quantum systems by laser radiation can be defined as a single input homogeneous in state, bilinear system. 
Solving it through gradient methods can turn problematic
because there are trajectory portions where there is no information about the recommended variation of the control trajectory. Krotov's method, however, does not suffer of this deficiency \mbox{\cite{95Krot}}. 
A slightly more generalized version of this class of problems was addressed for two forms of performance index. The first   is quadratic in the control \mbox{\cite{11Krot}} and the other is in the state  \mbox{\cite{20Hal}}.
Additional types of optimal bilinear control problems, solved by Krotov's method, 
can be found in the literature \mbox{\cite{19Hal,19Hal02}}.\\

This paper provides another contribution to optimal control of finite dimensional bilinear systems by generalizing previous works. Here a continuous time, bilinear state equation is concerned with a quadratic performance index to be minimized. The state equation is time varying, given in a general bilinear form and comprises a deterministic, a-priori known excitation. The control trajectory is constrained to an admissible set without a specific structure. The performance index is a functional, quadratic  in the state variables and control signals, and includes a terminal cost. 
Krotov's method is used here for solving this problem. The novelties include the derived sequence of improving  functions as well as some additional theoretical remarks.
Section \ref{sec:PF} defines the addressed problem. It is addressed by Krotov's method, which is described in section \ref{sec:BG}. The main results are described in section \ref{sec:MR} and a numerical example is given in section \ref{sec:NE}. Section \ref{sec:CL} concludes the paper.\\

Throughout this paper, bold lowercase notation is used  for vectors, e.g. $\ax=(x_i)_{i=1}^{n_u}$, normally column ones. Bold uppercase notation is used for matrices. A trajectory is some vector function, e.g. $\ax$ is a state trajectory $\bR\rightarrow\bR^{n_u}$.  $\ax(t)$ refers to the trajectory $\ax$ evaluated at $t$.

\section{Problem Statement}
\label{sec:PF}
A set of admissible control signals is denoted below as $\cU(\ax)$. 
Its state dependency implies that it might be involved in the constraints laid upon the control signals. 
\begin{definition}
	\label{def:CBQRAdmProce}
	Let $\ax:\bR\rightarrow\bR^n$ be a state trajectory and $\au:\bR\rightarrow\bR^{n_u}$ be a  control trajectory. 
	Let $\cU_i(\ax)$ be a set of control signals, admissible at a state trajectory $\ax$. 
	If the pair of trajectories---$(\ax,\au)$---satisfies the  bilinear state equation
	\ali{
		\eqgrp{
		\dot\ax(t) =& \aA(t)\ax(t) + \aB(t)\au(t) + \{\au\aN(t)\}\ax(t) + \ag(t)\\
		&\ax(0), \forall t\in(0,t_f)
		}
		 \label{eq:CBQR_StateEq}
	} 
	and $u_i\in \cU_i(\ax)$ for all $i=1,\ldots,n_u$ then
	$(\ax,\au)$ is called an 'admissible process'.
	Here 
	$\{\au\aN(t)\}\triangleq\sum_{i=1}^{n_u} u_i(t)\aN_i(t)$;
	$\aA,\aN_i:\bR\rightarrow\bR^{n\times n}$;
	$\aB:\bR\rightarrow\bR^{n\times n_u}$ and
	$\ag:\bR\rightarrow\bR^n$. $\ag$ is a trajectory of external excitations.
\end{definition}

The \textit{continuous time bilinear quadratic regulator} (CBQR) problem is defined as follows.
\begin{definition}[CBQR]
\label{def:CBQR} 
Find an optimal and admissible process, $(\ax^*,\au^*)$, that  minimizes the quadratic performance index
\ali{
	\eqgrp{
	J(\ax,\au) =& \left(\frac{1}{2}\int\limits_0^{t_f} 
		\ax(t)^T\aQ(t)\ax(t)  +  \au(t)^T\aR(t)\au(t) 
		\ud t \right) \\
		&+ \frac{1}{2}\ax(t_f)^T\aH\ax(t_f)
		\label{eq:CBQR_J}
		}
}
where 
$\aQ:\bR\rightarrow\bR^{n\times n}$ such that $\aQ(t)\geq 0$,
$\aR:\bR\rightarrow\bR^{n_u\times n_u}$ such that $\aR(t)\geq 0$ and $\aH\geq 0$.
\end{definition}

\section{Background - Krotov's method}
\label{sec:BG}

This section succinctly overviews portions of Krotov's theory, tailored and relevant to the treated problem. The definitions and theorem, stated here, are needed for understanding the results provided in section \ref{sec:MR}. \\

Let
\ali{
	\dot \ax(t) = \af&(t,\ax(t),\au(t));	
	\quad\ax(0),
	\forall t\in(0,t_f)
	 \label{eq:StSpGenSt_2}
	}
be a state equation; 
$\cU\subseteq\{\bR\rightarrow\bR^{n_u}\}$ be a set of admissible control trajectories 
and
$\cX\subseteq\{\bR\rightarrow\bR^{n}\}$ be a set of state trajectories that are reachable from $\cU$ and $\ax(0)$.
In what follows, 
the term 
\textit{process} refers to some pair $(\ax\in\cX,\au\in\cU)$. 
A pair $(\ax,\au)$   that satisfies \Eqref{eq:StSpGenSt_2} is called an \textit{admissible process}.
Let the performance index be a non-negative functional $J:\cX\times\cU\rightarrow[0,\infty)$.
$\af$ and $J$ define an optimal control problem, as follows.

\begin{definition}[Constrained Optimal Control Problem]
\label{def:ocprob}
Find an admissible process $(\ax^*,\au^*)$ that minimizes
\ali{
	J(\ax,\au) =& l_f(\ax(t_f)) + \int\limits_0^{t_f} l(t,\ax(t),\au(t)) \ud t \label{eq:OCP_J}
	}
\end{definition}
In some problems it is easier to embrace sequential approach, instead of directly trying to synthesize $(\ax^*,\au^*)$. In this context, a useful type of sequence is given below.
\begin{definition}[Improving Sequence (\cite{95Krot}, section 6.2)]
Let $\{(\ax_k,\au_k)\}$ be a sequence of  admissible processes and assume that $\inf_{\substack{\ax\in\cX\\\au\in\cU}} J(\ax,\au)$ exists. If
\ali{
	J(\ax_k,\au_k)\geq J(\ax_{k+1},\au_{k+1}) \label{eq:Krot_OptSeq}
}
for all $k=0,1,2,\ldots$ and:
\ali{
	\lim_{k\rightarrow \infty} J(\ax_k,\au_k) 
}
exists, then $\{(\ax_k,\au_k)\}$ is an \textit{improving sequence}.
\end{definition}
Krotov's method is aimed at computing such a sequence by  successively improving  admissible processes.   
This is done by repeatedly solving a sub-problem of process improving, as follows. 
In the next  derivations, $\axi$ and $\anu$ are some vectors in $\bR^n$ and $\bR^{n_u}$, respectively. $\cX(t)\subseteq{\bR^n}$ refers to an intersection of $\cX$ at $t$, i.e, the set of vectors obtained by evaluating the trajectories in $\cX$ at $t$. Let $q:\bR\times\bR^n\rightarrow\bR$ be a function of time and the state vector. $q_t$ stands for $q$'s partial derivative w.r.t. the time argument.  
$q_\ax$ is its  gradient w.r.t. the state vector.
In order to improve a given process it is suffice to find $q$ and a feedback $\hat \au$, meeting the requirements described in the next theorem.
\begin{theorem}[\cite{95Krot}, theorem 6.1]
\label{thm:krotmeth}
Let a given admissible process be $(\ax_1,\au_1)$ and let $q:\bR\times\bR^n\rightarrow\bR$. 
The functions $s:\bR\times\bR^n\times\bR^{n_u}\rightarrow\bR$ and $s_f:\bR^n\rightarrow\bR$ are constructed from $q$ according to:
\ali{
	s(t,\axi,\anu) \triangleq &  l(t,\axi,\anu) + q_t(t,\axi) + q_\ax(t,\axi) \af(t,\axi,\anu)  \label{eq:Krot_s}\\
	s_f(\axi) \triangleq  & l_f(\axi) - q(t_f,\axi) \label{eq:Krot_sf}
}
If the next three statements hold:
\begin{enumerate}
\item	$q$ grants $s$ and $s_{f}$ the property:
\ali{
\eqgrp{
s(t,\ax_1(t),\au_1(t)) =& \max_{\axi\in\cX(t)} s(t,\axi,\au_1(t));\quad \forall t\in(0,t_f)\\
	s_f(\ax_1(t_f)) =& \max_{\axi\in\cX(t_f)} s_{f}(\axi)
	}\label{eq:Krot_AlgMaxProb}
}
\item $\hat \au$ is a  feedback that satisfies 
\ali{
\hat \au(t,\axi) = \arg\min_{\anu\in\cU(t)} s(t,\axi,\anu) 
	\label{eq:Krot_AlgUMinProb}
}
for all $t\in[0,t_f],\axi\in\cX(t)$.
\item $\ax_2$ is a state trajectory  that solves:
\ali{
	&\dot \ax_2(t) = \af(t,\ax_2(t),\hat \au(t,\ax_2(t)));
	\quad\ax_2(0) = \ax(0),
		 \label{eq:krotmeth_ss_imp}
	}
at any $\forall t\in(0,t_f)$, and $\au_2$ is a control trajectory such that $\au_2(t)=\hat \au(t,\ax_2(t))$.
\end{enumerate}
then $(\ax_2,\au_2)$ is an improved process. 
\end{theorem}

The method repeats process improvements over and over, thereby requiring to derive a sequence of improving functions---$\{q_k\}$. If such a sequence can be found, it allows to compute a process that is a candidate global optimum.\\

Generally, the procedure is summarized in the following algorithm.
Hereinafter, $s_k$ and $s_{f,k}$ are used for signifying $s$ and $s_f$ that are constructed by $q_k$.
First, some initial admissible process---$(\ax_0,\au_0)$, should be computed. 
Afterwards, the following steps are iterated for $k=\{0, 1, 2,\ldots\}$:
\begin{enumerate}
\item 
	Find $q_k$ that grants $s_k$ and $s_{f,k}$ the property
	\aliu{
	\eqgrp{
	s_k(t,\ax_k(t),\au_k(t)) =& \max_{\axi\in\cX(t)} s_k(t,\axi,\au_k(t))\\
		s_{f,k}(\ax_k(t_f)) =& \max_{\axi\in\cX(t_f)} s_{f,k}(\axi)
	}
	}
	at the current $(\ax_k,\au_k)$ and for all $t$ in $[0,t_f)$. 
\item 
	Find an improving feedback $\hat\au_{k+1}$ such that
	\aliu{
		\hat\au_{k+1}(t,\ax(t)) = \arg\min_{\anu\in\cU(t)} s_k(t,\ax(t),\anu)
	}
	for all $t$ in $[0,t_f]$
\item   Compute an improved process by solving
	\aliu{
		\dot\ax_{k+1}(t) =& \af{\big(}t,\ax_{k+1}(t),\hat\au_{k+1}(t,\ax_{k+1}(t)){\big)}
	}
	and setting $\au_{k+1}(t) = \hat\au_{k+1}(t,\ax_{k+1}(t))$.
\end{enumerate}

\begin{myremarks}
\item This algorithm generates an improving sequence of processes 
$\{(\ax_k,\au_k)\}$  such that $J(\ax_k,\au_k)\geq J(\ax_{k+1},\au_{k+1})$ \cite{95Krot}. 
\item If at some point, the processes in the improving sequence stop changing, then the process satisfies Pontryagin's minimum principle \cite{95Krot}, thereby inferring that it is a candidate optimum of the given problem.
\item The method has a significant advantage over algorithms that are based on small variations. 
The latter are constrained to small process variations, which is troublesome as: (1) it leads to slow convergence rate  and (2) for some optimal control problems small variations are impossible \cite{88Krot}. 
\end{myremarks}

\section{Main Results}
\label{sec:MR}
In order to solve the CBQR problem by Krotov's method, a class of suitable improving functions 
 should be formulated. Such a class is derived in the next lemmas.
 The notation $\aM(t,\axi)$ signifies an $\bR^{n\times n_u}$ matrix---
$\mtrx{\aN_1(t)\axi & \aN_2(t)\axi & \ldots \aN_{n_u}(t)\axi }$.
\begin{lemma}
\label{CBQR:lemma7}
Let 
\aliu{
	q(t,\axi) = \frac{1}{2}\axi^T\aP(t)\axi + \ap(t)^T\axi;\quad \aP(t_f) = \aH;\,\ap(t_f)=\aZe
}
where $\axi\in\bR^n$, $\aP:\bR\rightarrow\bR^{n\times n}$ is a smooth and symmetric, matrix function   and $\ap:\bR\rightarrow\bR^{n}$ is a smooth, vector function.\\

The vector of control laws, $(\hat u_i)_{i=1}^{n_u}$, that minimizes $s(t,\ax(t),\au(t))$, is given by
\ali{
\eqgrp{
\hat u_i(t,\ax(t)) =& \arg \min_{\anu\in\cU_i(t,\ax)} \Biggl(\frac{1}{2}\anu^T\aR(t)\anu\\
				&\hspace{1cm}+q_\ax(t,\ax(t))(\aB(t)+\aM(t,\ax(t)))\anu\Biggr)
}
		  \label{eq:CBQRargmin_u_s}
}
\end{lemma}
\begin{proof}
Let $\anu\in\bR^{n_u}$. By substituting the state equation into $s$, we get:
\ali{
	\eqgrp{s(t&,\ax(t),\anu)	
			=q_t(t,\ax(t)) + q_\ax(t,\ax(t)) \af(t,\ax(t),\anu)\\
			&\hspace{3cm}+ \frac{1}{2} \left(\ax(t)^T\aQ(t)\ax(t)+ \anu^T\aR(t)\anu
			\right)} \notag\\
	\eqgrp{
		=&q_t(t,\ax(t)) 
		+ q_\ax(t,\ax(t)) \Bigl(\aA(t)\ax(t) + \aB(t)\anu \\
		&  \hspace{3.5cm}+ \{\anu\aN(t)\}\ax(t) + \ag(t)\Bigr)\\
		&+ \frac{1}{2}\ax(t)^T\aQ(t)\ax(t)+ \frac{1}{2}\anu^T\aR(t)\anu} \label{eq:CBQRExc_s1MIMO}\\
	 \eqgrp{ 
	 	=&f_2(t,\ax(t))+q_\ax(t,\ax(t))(\aB(t)+\aM(t,\ax(t)))\anu +\frac{1}{2}\anu^T\aR(t)\anu} \notag
}
where $f_2:\bR\times\bR^n\rightarrow\bR$ is some mapping, independent of $\anu$. It follows that a minimum of $s(t,\ax(t),\anu)$ over $\{\anu|\anu\in\cU(t,\ax)\}$  depends merely on the two last terms.
\end{proof}
A key part of the improving step is to find $q$ such that \Eqref{eq:Krot_AlgMaxProb} will hold.
In this paper the strategy is finding $q$ such that $s$ will no longer be state dependent. 
	This can be done thanks to the fact that when the control input is held fixed, \Eqref{eq:CBQR_StateEq} is affine w.r.t the state. Consequently \Eqref{eq:Krot_s}, induced by the suggested $q$, remains quadratic w.r.t the states. 
	Hence, it is possible to find $\aP$ and $\ap$ that remove the state-dependency from $s$, making the property in \Eqref{eq:Krot_AlgMaxProb} true. This is the gist of the next lemma.
\begin{lemma}
\label{CBQR:lemma8}
Let $(\ax,\au)$ be a given process, and let $\aP$ and $\ap$  be solutions to
\ali{
\eqgrp{
\dot\aP(t) =& - \aP(t)(\aA(t) +\{\au\aN(t)\}) \\
		&\hspace{1cm}- (\aA(t) +\{\au\aN(t)\})^T\aP(t) - \aQ(t)}
 	\label{eq:CBQRKrot_maxP}
}
with $\aP(t_f)=\aH$, and
\ali{
\eqgrp{
 \dot \ap(t) = & - \left(\aA(t) +\{\au\aN(t)\}\right)^T\ap(t) \\
 		&\hspace{1cm}-\aP(t)(\aB(t)\au(t) + \ag(t))}
\label{eq:CBQRKrot_maxp}
}
with $\ap(t_f)=\aZe$. Then 
\aliu{
	q(t,\ax(t)) =& \frac{1}{2}\ax(t)^T\aP(t)\ax(t) + \ap(t)^T\ax(t)
}
grants $s$ and $s_f$ the property: 
\ali{
	s(t,\ax(t),\au(t)) =& \max_{\axi\in\cX(t)} s(t,\axi,\au(t))\label{eq:CBQR_smax} \\
	s_f(\ax(t_f)) =& \max_{\axi\in\cX(t_f)} s_f(\axi)
}
\end{lemma}
\begin{proof}
For any $\ax(t_f)\in\cX(t_f)$ we get
\aliu{
s_f(&\ax(t_f))= \frac{1}{2}\ax(t_f)^T\aH\ax(t_f) - q(t_f,\ax(t_f)) = 0
}
therefore it is obvious that $s_f(\axi) \leq s_f(\ax(t_f))$ for all $\axi\in\cX(t_f)$. 
The partial derivatives of $q$ are 
\aliu{
q_t(t,\axi) =&  \frac{1}{2}\axi^T\dot\aP(t)\axi + \dot \ap(t)^T\axi\\
q_\ax(t,\axi) =& \axi^T\aP(t) + \ap(t)^T
}
As $\aP(t)$ and  $\ap(t)$ satisfy \Eqsref{eq:CBQRKrot_maxP} and \eqref{eq:CBQRKrot_maxp}, 
substituting $q_t$ and $q_\ax$ into \Eqref{eq:CBQRExc_s1MIMO} and then arranging it in a canonical form for $\ax(t)$, yields:
\ali{
	\eqgrp{s(t,\ax&(t),\au(t))\\
	 	=&\frac{1}{2}\ax(t)^T \Bigl(\dot\aP(t) 	
	 	  + \aP(t)\left(\aA(t)+\{\au\aN(t)\}\right) \\
	 	 	& \hspace{0.5cm}+ \left(\aA(t)+\{\au\aN(t)\}\right)^T\aP(t) + \aQ(t)\Bigr)\ax(t)\\
	 	 &+\ax(t)^T\Bigl( \dot \ap(t) + \left(\aA(t)+\{\au\aN(t)\}\right)^T\ap(t) \\
	 	 	&\hspace{2.5cm}+\aP(t)(\aB(t)\au(t) + \ag(t))\Bigr)\\ 
	 	 & +\ap(t)^T(\aB(t)\au(t) + \ag(t)) 
	 	 		+  \frac{1}{2}\au(t)^T\aR(t)\au(t)}	\notag\\
		=&\ap(t)^T(\aB(t)\au(t) + \ag(t))  +  \frac{1}{2}\au(t)^T\aR(t)\au(t) \label{CBQR:lemma8:eq0}
}
Because the dependency on $\ax(t)$ has dropped, it is obvious that $s(t,\ax(t),\au(t)) = s(t,\axi,\au(t))$, and that
\aliu{
	s(t,\axi,\au(t))) \leq  s(t,\ax(t),\au(t)) 
}
for all $\axi\in\cX(t)$.
\end{proof}

Let $(\ax_1,\au_1)$ be a given admissible process. Solve \Eqsref{eq:CBQRKrot_maxP} and \eqref{eq:CBQRKrot_maxp} to $\au_1$ and obtain $\aP_1$ and $\ap_1$. 
An improving feedback---$\hat \au_2$, which  is constructed by \Eqref{eq:CBQRargmin_u_s} in conjunction with $\aP_1$ and $\ap_1$, allows to obtain an improved process - $(\ax_2,\au_2)$.\\

In other words, lemmas \ref{CBQR:lemma7} and \ref{CBQR:lemma8} allow to compute two sequences: $\{q_k\}$ and $\{(\ax_k,\au_k)\}$  such that the second one is an improving sequence. 
As $J$ is non-negative, it has an infimum and $\{J(\ax_k,\au_k)\}$ converges. In case that the processes stop changing, then a candidate optimum is obtained (see the second remark to theorem \ref{thm:krotmeth}).
\\

The corresponding algorithm  is summarized in \figref{fig:CBQRKrotAlg}.
Its output is an arbitrarily close approximation of  $\aP^*$ and $\ap^*$, which define the candidate optimal control law (\Eqref{eq:CBQRargmin_u_s}) through $q^*$. It should be noted that the use of absolute value in step (4) of the iterations stage is theoretically unnecessary. Though, due to numerical computation errors, the algorithm might lose its monotonicity when the processes are getting closer to a candidate optimum.\\
\begin{myremarks}
\item \Eqref{CBQR:lemma8:eq0} provides an alternative approach for computing  $J(\ax,\au)$. 
	As $s_f(\ax(t_f))=0$, it can be shown that 
\aliu{
	J(\ax,\au) =& J_{eq}(\ax,\au) \triangleq q(0,\ax(0)) 
		+ \int\limits_0^{t_f} s(t,\ax(t),\au(t)) \ud t\notag\\
		\eqgrp{ =& \frac{1}{2}\ax(0)^T\aP(0)\ax(0) + \ap(0)^T\ax(0) \\
			&+\Biggl(\int\limits_0^{t_f}  \ap(t)^T(\aB(t)\au(t) + \ag(t))  \\
				&\hspace{2cm}+  \frac{1}{2}\au(t)^T\aR(t)\au(t)
				\ud t\Biggr)
		}
}
\item	The minimization described in \Eqref{eq:CBQRargmin_u_s} defines the improving feedback. It should be solved for each $t\in[0,t_f]$ independently and can have different attributes, depending on $\cU$ and $\aR$. E.g., Let $\aR(t)\equiv \aR_0>0$ and $\cU(t,\ax)=\bR^{n_u}, \forall t\in[0,t_f]$. 
Then 
\aliu{
		\hat \au(t,\ax(t))  =&  
			\aR_0^{-1}(\aB(t)+\aM(t,\ax(t)))^T\left(\aP(t)\ax(t) + \ap(t)\right)
		}

	Another example is the singular case, e.g. when
 $\aR(t)=\aZe$ at some $t$. Then, in this time instance
the improving feedback is 
\aliu{
		\hat \au(t,\ax(t))  =& \arg \min_{\anu\in\cU_i(t,\ax)} 
			\Bigl((\ax(t)^T\aP(t) + \ap(t)^T)\\
				&\hspace{2cm}\times(\aB(t)+\aM(t,\ax(t)))\anu \Bigr) 
		}
This is in general a non-linear programming problem\cite{84Leun}. An existence of solution, in this case, depends on the nature of $\cU(t,\ax)$ and therefore should be discussed in the context of a given problem. 
\item 
Treat $\au$ as a given trajectory. Consequently, \Eqref{eq:CBQRKrot_maxP} is a linear matrix ODE with respect to $\ax$. Using Kronecker product, it can be reorganized into a standard linear ODE form \mbox{\cite{59Bellman}}.
Next, assume that the elements of $\aA$, $\aB$, $(\aN_i)_{i=1}^{n_u}$, $\ag$ and $\aQ$  are all integrable on $(0,t_f)$. 
If $\au$ is integrable on $(0,t_f)$ then this linear ODE meets Caratheodory's conditions. Thereby it has a unique absolutely continuous solution---$\aP$ \mbox{\cite{80Hale}}.
  As the domain is finite and $\aP$ is absolutely continuous then $\aP$ is bounded. It follows that $\aP(\aB\au+\ag)$ is integrable on $(0,t_f)$ and \Eqref{eq:CBQRKrot_maxp} is a linear ODE that meets Caratheodory's conditions too. Hence there exists a unique bounded solution---$\ap$. 
\end{myremarks}
\begin{figure*}
\begin{framed}
\noindent\textbf{Input} \\
	$\aA$, $\aB$, $\{\aN_i\}$, $\ag$, $\cU$, $\ax(0)$, $\aQ$, $\aR$, $\aH$.\\
\textbf{Initialization}
	\begin{enumerate}
	\item Select a convergence tolerance - $\epsilon>0$.
	\item Provide some admissible process $(\ax_0,\au_0)$.
		Solve
	\aliu{
	\dot\aP_0(t) =& - \aP_0(t)(\aA(t) +\{\au_0\aN(t)\}) - (\aA(t) +\{\au_0\aN(t)\})^T\aP_0(t) - \aQ(t)
	;&&
	\aP_0(t_f)   = \aH\\
	 \dot \ap_0(t) = & - \left(\aA(t) +\{\au_0\aN(t)\}\right)^T\ap_0(t)-\aP_0(t)(\aB(t)\au_0(t) + \ag(t))
	 ;&& 
	 \ap_0(t_f)=\aZe
	}
	\item Compute 
		\aliu{
			J_0(\ax_0,\au_0) = \frac{1}{2}\int\limits_0^{t_f} \ax_0(t)^T\aQ(t)\ax_0(t) 
			+\au_0(t)^T\aR(t)\au_0(t)\ud t + \frac{1}{2}\ax(t_f)^T\aH\ax(t_f)
			}
	\end{enumerate}
\textbf{Iterations}\\
	For $k=\{0,1,2,\ldots\}$:
	\begin{enumerate}
		\item 	Formulate an improved process by solving
		\aliu{
		 \dot \ax_{k+1}(t) =
		 &\aA(t)\ax_{k+1}(t) + \aB(t)\hat \au_{k+1}(t,\ax_{k+1}(t)) 
		 + \{\hat \au_{k+1}(\ax_{k+1}(t))\aN(t)\}\ax_{k+1}(t) + \ag(t) }
		to $\ax_{k+1}(0) = \ax(0)$, where
		\aliu{
		\hat \au_{k+1}(t,\ax(t))  =& \arg \min_{\anu\in\cU_i(t,\ax)} \ay_k(t,\ax(t))\anu  
									+\frac{1}{2}\anu^T\aR(t)\anu\\
		\ay_k(t,\axi) =& (\axi^T\aP_k(t) + \ap_k(t)^T)(\aB(t)+\aM(t,\axi))
		}
			Set $\au_{k+1}(t) =\hat \au_{k+1}(t,\ax_{k+1}(t))$.
		\item Solve
			\aliu{
		\dot\aP_{k+1}(t) =& - \aP_{k+1}(t)(\aA(t) + \{\au_{k+1}\aN(t)\}) 
				- (\aA(t) + \{\au_{k+1}\aN(t)\})^T\aP_{k+1}(t) - \aQ(t) \\
		\dot \ap_{k+1}(t) = & - \left(\aA(t) + \{\au_{k+1}\aN(t)\}\right)^T\ap_{k+1}(t) 
								-\aP_{k+1}(t)(\aB(t)\au_{k+1}(t) + \ag(t))
		}
		 to $\aP_{k+1}(t_f)=\aH$ and $\ap_{k+1}(t_f)=\aZe$.
		\item		Compute
		\aliu{
		J(\ax_{k+1},\au_{k+1}) =& \frac{1}{2}\int\limits_0^{t_f} \ax_{k+1}(t)^T\aQ(t)\ax_{k+1}(t) 
			+\au_{k+1}(t)^T\aR(t)\au_{k+1}(t)\ud t + \frac{1}{2}\ax_{k+1}(t_f)^T\aH\ax_{k+1}(t_f)
		}
		\item 	If $|J(\ax_k,\au_k)-J(\ax_{k+1},\au_{k+1})|<\epsilon$, stop iterating. Otherwise - continue.
	\end{enumerate}
\textbf{Output} \\ $\aP_{k+1}$, $\ap_{k+1}$.
\end{framed}
\caption{CBQR - Algorithm for successive improvement of control process.\label{fig:CBQRKrotAlg}}
\end{figure*}

\section{Numerical Example}
\label{sec:NE}
\begin{figure}
	\begin{center}
	\includegraphics{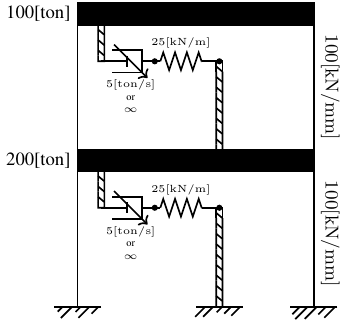}
	\end{center}
\caption{A dynamic scheme of a two floors structure, configured with two SAVS devices. \label{fig:krotmeth:cbqr:savs:dynschem}}
\end{figure}
The use of this method is exemplified here by a structural control problem. 
The method is applied to an optimal control problem of a structure with two degrees of freedom and two SAVS devices \cite{93Kobori}. Shear model is used for plant modeling. Its dynamic scheme and the devices are described in \figref{fig:krotmeth:cbqr:savs:dynschem}. The masses are lumped and located at the floors' ceilings. They  are $100\ton$ and $200\ton$, at the first and second floors, respectively. The floors' horizontal stiffnesses are identical and equal to $100\e{6}\nOmt$. Let $z_1$ and $z_2$ be the horizontal DOF in the ceilings. 
The related mass and stiffness matrices are:
\aliu{
\aM = \mtrx{100 & 0 \\ 0 & 200}\e{3};\;
\aK = \mtrx{200 & -100 \\ -100 & 100}\e{6};
} 
The structure is assumed to have inherent damping, described by a proportional damping matrix $\aC_d=0.001\aM+0.0001\aK$. A horizontal seismic external excitation was applied to the structure by ground acceleration signal of $\ddot z_g(t) = 3\sin(16.55 t)$. 
The equation of motion is:
\ali{
\aM\ddot\az(t) + \aC_d\dot\az(t) + \aK\az(t) = -\agam\ddot z_g(t) + \aPsi \aw(t) 
}
with:
\aliu{
\aPsi = \mtrx{1& -1 \\ 0 & 1}; \quad 
\gamma = \mtrx{1\\ 1}
}
Here $\az \triangleq (z_1 , z_2)$ is the trajectory of the DOF displacements;
$\aw\triangleq  (w_1 , w_2)$ is the trajectory of the control forces. $w_i(t)>0$ means that the force in the i\textsuperscript{th} device is a tension force. 
Spring and variable dashpot were used for modeling each SAVS device. The device's spring stiffness is 
$k_{savs}=25\e{6}\nOmt$, accounting for the bracing stiffness and the fluid's bulk modulus. The variable dashpot can provide either a finite damping coefficient of $5\e{3}\kgOsec$ or infinity, depending whether it is unlocked or locked. Here, the control policy is restricted to one out of three locking patterns\cite{93Kobori}---(1) both devices are unlocked, (2) only the second device is unlocked or (3) both devices are locked. Each device provides another state variable to the system. Here these states variables are represented by the control forces' signals, providing another two state equations:
\aliu{
\dot \aw(t) =& -5\e{3}\left(u_1(t)\mtrx{0 & 0 \\ 0 & 1} + u_2(t)\mtrx{1 & 0 \\ 0 & 1}\right)\aw(t)\\ 			&-25\e{6}\mtrx{1&0\\-1&1}\dot\az(t)
}
where $\au\triangleq (u_1 , u_2)$ is the control input trajectory. These settings lead to an admissible set of control inputs
\aliu{
\au(t) \in \cU = \left\{\mtrx{0\\1},\mtrx{1\\0},\mtrx{0\\0}\right\}
}  
Every object in this set reflects the SAVS devices' unlocking patterns---1,2 and 3---respectively. Note that $\cU$'s finiteness prohibits the use of variational methods with respect to $\au$.\\

Putting these equations altogether boils down to the bilinear state equation: 
\aliu{
&\dot\ax(t) = \aA\ax(t) + u_1(t)\aN_1\ax(t) + u_2(t)\aN_2\ax(t) + \ag(t)\\
&\aA = \mtrx{	\aZe & \aI & \aZe \\ 
	-\aM^{-1}\aK & -\aM^{-1}\aC & \aM^{-1}\aPsi \\
	\aZe & -\aPsi^T/k_{savs} &  \aZe 
	}\\
&\aN_1 = \diag(0,0,0,0,0,-5000)\\
&\aN_2 = \diag(0,0,0,0,-5000,-5000)\\
&\ag(t) = \mtrx{0&0&-3&-3&0&0}^T\sin(16.55t)
}
where the state trajectory is $\ax=(z_1 , z_2 , \dot z_1 , \dot z_2 , w_1 , w_2)$. The time span for the problem was set to $(0,5)$ and the initial conditions were set to zero.\\

The performance evaluation accounts for inter-story drifts and control forces. It is:
\aliu{
J(\ax,\au)& = \Biggl(\frac{1}{2}\int\limits_0^5 x_1(t)^2\e{5} 
			+ (x_2(t)-x_1(t))^2\e{5} + 5(x_5(t)^2 \\
			&+ x_6(t)^2) \ud t\Biggr)+\frac{1}{2}\Biggl(x_1(5)^2\e{4}+ (x_2(5)-x_1(5))^2\e{4}\\
		&+ 50(x_5(5)^2 + x_6(5)^2)\Biggr)
}
$\aQ$ and $\aH$ were constructed accordingly. It follows that $\aR\equiv\aZe$ as 
the control inputs have no weight in $J$.\\

According to step 1 of the iterations stage, an improving feedback should be found at each time instance. In the addressed problem, it reduces to \[\hat \au_{k+1}(t,\ax(t))  = \arg \min_{\anu\in\cU} \ay_k(t,\ax(t))\anu\] where $\ay_k$ is row vector, defined in the algorithm. This minimization should be carried  over a finite set of points and therefore can be calculated by a simple table-search technique.\\

The method was realized by MATLAB through a standard desktop PC. Solving the ODEs was based on a 4\textsuperscript{th} order Runge-Kutta method.
Two cases were computed. 
The first case is a straightforward application of the suggested method. It is referred to as CBQR. 
The second one, referred to as MPC, describes a \textit{model-predicitve-control} with a CBQR optimizer. Its horizon was set to 1 second and 3 improvements were carried out in each control update.
In the CBQR case the improvement process was iterated 15 times, each lasting approx 5.9 seconds in average. 
Table \ref{tab:krotmeth:cbqr:savs:J} provides $J$'s values of the processes obtained in this case. In this table $i$ stands for the process number. The table also provides the relative change in $J$, signified by $\Delta J_i\triangleq J_i-J_{i-1}$. $J$'s value at the MPC process is $1.78\e{13}$. 
\begin{table} 
	\centering
	\caption{Peformance index values at each process. Here $\Delta J_i = J_i-J_{i-1}$
		\label{tab:krotmeth:cbqr:savs:J}.}
	\label{tab:Jconverge}
\begin{tabular}{ccc|ccc}  
	$i$ & $J$ & $\Delta J$ & $i$ & $J$ & $\Delta J$ \\
0 & $4.59\e{14}$ & - & 8 & $4.51\e{12}$ & $-2.87\e{10}$\\
1 & $1.54\e{14}$ & $-3.05\e{14}$ & 9 & $4.49\e{12}$ & $-1.38\e{10}$\\
2 & $1.44\e{13}$ & $-1.39\e{14}$ & 10 & $4.49\e{12}$ & $-5.34\e{9}$\\
3 & $5.54\e{12}$ & $-8.9\e{12}$ & 11 & $4.48\e{12}$ & $-8.91\e{9}$\\
4 & $4.66\e{12}$ & $-8.81\e{11}$ & 12 & $4.47\e{12}$ & $-8.72\e{9}$\\
5 & $4.65\e{12}$ & $-4.5\e{9}$ & 13 & $4.49\e{12}$ & $1.92\e{10}$\\
6 & $4.6\e{12}$ & $-4.9\e{10}$ & 14 & $4.47\e{12}$ & $-1.41\e{10}$\\
7 & $4.53\e{12}$ & $-6.9\e{10}$ & 15 & $4.47\e{12}$ & $-7.55\e{9}$\\

\end{tabular}
\end{table} 
The form of the control signals, synthesized by the feedback obtained in the CBQR case, are showed in \figref{fig:krotmeth:cbqr:savs:u}. As the signals are rapidly varying they are presented over a representing time interval $t\in[0.5,1.2]$. Careful inspection of $\au$ will tell that $u_1(t)u_2(t)=0$ for all $t$, meaning that only one unlocking pattern is active in each moment, as it was required  by the problem's definition.
Generally, SAVS controllers fall into one out of two operating modes---\textit{resetting mode} or \textit{switching mode}. The former refers to devices that are normally closed and are opened momentarily. The latter is more general, referring to  devices that are switching from open to close and vice-versa, without favoring any of these states \cite{01Nasu}. 
\figref{fig:krotmeth:cbqr:savs:u} presents finite unlocking segments, thereby manifesting that switching mode is the optimal in this case.
\begin{figure} 
\centering
\subfloat[$u_1$---unlocking pattern 1: only device 2 is unlocked]{
\centering
\includegraphics{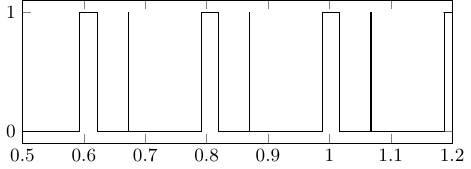}
}\\[5mm]
\subfloat[$u_2$---unlocking pattern 2: both devices are unlocked]{
\centering
\includegraphics{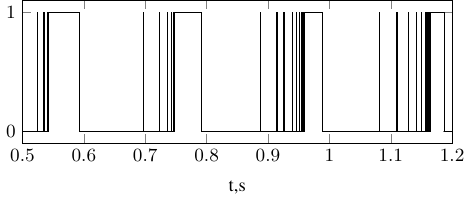}
}
\caption{ Control signals of the CBQR case.\label{fig:krotmeth:cbqr:savs:u}}
\end{figure} 
\figref{fig:krotmeth:cbqr:savs:w} presents the control force signals of each case. When a device is unlocked, the elastic energy that it gradually accrued during its locked state, rapidly drops and zeros the control force in the device. This is the reason for $\aw$'s saw-tooth shape. Each of these drops is related with one of the unlocking patterns, related to $\au$. 
\begin{figure} 
\centering
\subfloat[]{
\centering
\includegraphics{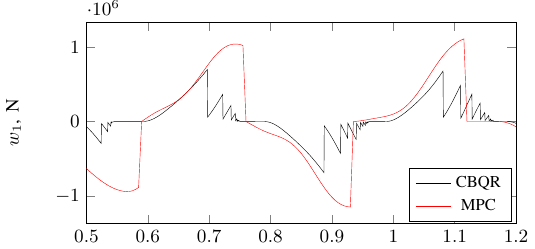}
}\\[5mm]
\subfloat[]{
\centering
\includegraphics{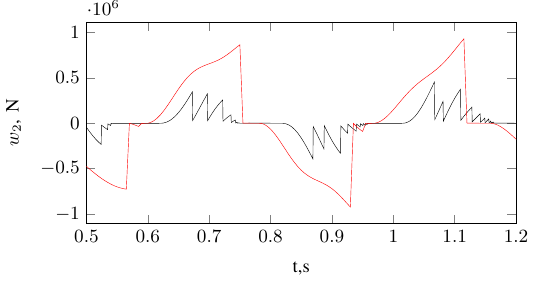}
}
\caption{Control force signals (states $x_5$ and $x_6$) \label{fig:krotmeth:cbqr:savs:w}.}
\end{figure} 
An effective control is expected to reduce the amplitude of the vibrating structure during the ground motion. This goal was indeed attained by each of the synthesized controllers, as shown in \figref{fig:krotmeth:cbqr:savs:z}.
\begin{figure} 
\centering
\subfloat[]{
\centering
\includegraphics{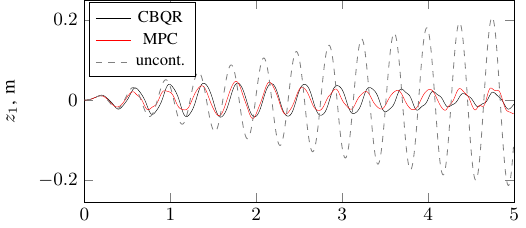}
}\\[5mm]
\subfloat[]{
\centering
\includegraphics{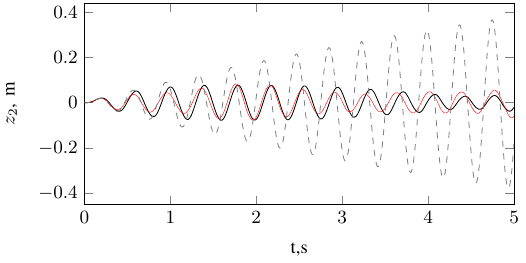}
}
\caption{ Displacements at the DOFs (states $x_1$ and $x_2$) \label{fig:krotmeth:cbqr:savs:z}.}
\end{figure}

\section{Conclusion}
\label{sec:CL}
The continuous time bilinear quadratic regulator is an optimal control problem, concerning a finite dimensional, continuous time bilinear state equation and a quadratic performance index to be minimized. In this paper the state equation is time varying, given in a general bilinear form and comprises a deterministic, a-priori known excitation. The control trajectory is constrained to an admissible set that is left here in a general form. This form should be set in accordance to the nature of a specific/class of addressed CBQR problem. The performance index is a functional, quadratic  in the state variables and control signals, and includes a terminal cost. This problem, which is written here in general structure, is addressed by Krotov's method. The main results comprises a definition of a suitable sequence of improving functions, allowing to generate an improving sequence of processes. This sequence  is convergent by means of the performance index. 
If the processes converge too, then the obtained process 
provides an arbitrarily close approximation to a candidate optimal feedback. The trajectory, synthesized by this feedback, satisfies Pontryagin's minimum principle. Even though it is merely a necessary condition, it is normally enough in many practical applications. 
As a matter of fact, assuming that the improving sequence, obtained by Krotov's method, consists of at least two different elements and the processes converge, then it is not only that the outcome satisfies Pontryagin's minimum principle but it is also guaranteed not being  maximal \cite{95Krot}. Solving by this method, which is not relying on small variations, is a significant advantage over such that do. It makes the suggested solution valid in problems with control inputs that don't allow small variations, e.g. control vectors that are constrained to closed subsets of $\bR^{n_u}$.  
Finally, the paper summarizes the results in an algorithm and furnishes a numerical example of structural control. The example consists of a structure with two floors that is controlled by two semi-active, variable-stiffness devices. This configuration constrains the control signals to a discrete set of values, hence prohibiting the use of methods that rely on small variations of the control signals.
Two control solutions are presented to this problem---a formal CBQR solution as well as an MPC with a CBQR optimizer.
\\

These results provide a useful tool for control design of bilinear systems and/or as a benchmark for other control strategies for such systems. 
  

\end{document}